%% file: SS4D.tex
\documentclass{amsart}

\def\titlepaper{Semisimple 4-dimensional topological field theories cannot detect exotic smooth structure}

\include{SS4D-preamble}

\def\handlegap{\hspace{0.3cm}}

\title{Semisimple $4$-dimensional topological field theories cannot detect exotic smooth structure}

\author{David Reutter}
\address{Max Planck Institute for Mathematics}
\email{reutter@mpim-bonn.mpg.de}
\urladdr{https://www.davidreutter.com}

\newcommand\ignore[1]{}

\begin{document}

\begin{abstract}
We prove that semisimple $4$-dimensional oriented topological field theories lead to stable diffeomorphism invariants and can therefore not distinguish homeomorphic closed oriented smooth $4$-manifolds and homotopy equivalent simply connected closed oriented smooth $4$-manifolds. We show that all currently known $4$-dimensional field theories are semisimple, including unitary field theories, and once-extended field theories which assign algebras or linear categories to $2$-manifolds. As an application, we compute the value of a semisimple field theory on a simply connected closed oriented $4$-manifold in terms of its Euler characteristic and signature.

Moreover, we show that a semisimple $4$-dimensional field theory is invariant under $\CPt$-stable diffeomorphisms if and only if the Gluck twist acts trivially. This may be interpreted as the absence of fermions amongst the `point particles' of the field theory. Such fermion-free field theories cannot distinguish homotopy equivalent $4$-manifolds. 

Throughout, we illustrate our results with the Crane-Yetter-Kauffman field theory associated to a ribbon fusion category, settling in the negative the question of whether it is sensitive to smooth structure. As a purely algebraic corollary of our results applied to this field theory, we show that a ribbon fusion category contains a fermionic object if and only if its Gauss sums vanish.
\end{abstract}

\maketitle	

\section{Introduction}

\subsection{Summary of results} 
Motivated  by a wealth of powerful field-theoretically-inspired $4$-manifold invariants~\cite{donaldson, witten, ozsvathszabo, kronheimermrowka}, a major open problem in quantum topology is the construction of a $4$-dimensional topological field theory in the sense of Atiyah-Segal~\cite{Atiyah,Segal} which is sensitive to exotic smooth structure. In this paper, we prove that no \emph{semisimple} topological field theory (Definition~\ref{def:semisimple}) can achieve this goal. Every currently known example of a full $4$-dimensional oriented topological field theory is semisimple and hence subject to our results, including invertible field theories (Example~\ref{exm:invertible}), unitary field theories (Theorem~\ref{thm:unitaryss}), and once-extended field theories (Theorem~\ref{thm:extended}) with values in any of the symmetric monoidal bicategories appearing in the `bestiary of $2$-vector spaces' of~\cite[App A]{III}, such as
\begin{itemize}
\item[--]  the bicategory of algebras, bimodules and bimodule maps;
\item[--] the bicategory of additive and idempotent complete linear categories, linear functors and natural transformations.
\end{itemize}

\noindent Concretely, we prove that semisimple field theories lead to stable diffeomorphism invariants.\looseness=-2
\begin{maintheorem}\label{thm:main}Let $Z$ be a semisimple oriented $4$-dimensional topological field theory and let $W$ and $W'$ be $S^2\times S^2$-stably diffeomorphic\footnote{Two connected compact oriented $4$-bordisms $W, W':M\to N$ are \emph{$S^2\times S^2$-stably diffeomorphic} if there is an integer $n\in \mathbb{Z}_{\geq 0}$ and an orientation-preserving diffeomorphism of bordisms between the $n$-fold connected sums $W\#^n(S^2\times S^2)$ and $W'\#^n(S^2\times S^2)$, where the connected sum is taken in the interior of the bordisms.} connected compact oriented $4$-bordisms.\\ Then $Z(W) = Z(W')$.
\end{maintheorem}
Theorem~\ref{thm:main} is proven in Section~\ref{sec:S2S2stability} by decomposing $Z$ into a finite direct sum of indecomposable theories which are multiplicative under connected sum (Proposition~\ref{prop:multiplicativityconnectsum}) and invertible on $S^2\times S^2$ (Theorem~\ref{thm:S2invertible}).

Using a theorem of Gompf~\cite{Gompf} and the classification of stable diffeomorphism classes of $4$-manifolds~\cite{Wall, Kreck}, we obtain our main theorem as a corollary of Theorem~\ref{thm:main}.%
\begin{maincor}\label{cor:main} Let $Z$ be a semisimple oriented $4$-dimensional topological field theory and let $M$ and $N$ be closed oriented $4$-manifolds.
\begin{enumerate}
\item If there is an orientation-preserving homeomorphism $M\to N$, then $Z(M) = Z(N)$.
\item If $M$ and $N$ are simply connected and if there is an orientation-preserving homotopy equivalence $M\to N$, then $Z(M)= Z(N)$.
\item If there is an orientation-preserving homotopy equivalence $M\to N$, and if the universal covers of $M$ and $N$ do not admit spin structures, then $Z(M) = Z(N)$.

More precisely, for any connected closed oriented $4$-manifold whose universal cover does not admit a spin structure, $Z(M)$ only depends on the Euler characteristic $\chi(M)$, the signature $\sigma(M)$, the fundamental group $\pi_1(M)$ and the image of the fundamental class $c_*[M]\in H_4(\pi_1(M), \mathbb{Z})$ under a classifying map $c:M\to K(\pi_1(M), 1)$ of the universal cover.
\end{enumerate}
\end{maincor}
Corollary~\ref{cor:main}, proven in Section~\ref{sec:S2S2stability}, is in marked contrast to the $3$-dimensional situation where semisimple topological field theories such as the Witten-Reshetikhin-Turaev field theory~\cite{RT} can distinguish certain homotopy equivalent lens spaces~\cite{FreedGompf}.

Using Theorem~\ref{thm:main}, we may evaluate a semisimple field theory on a closed oriented $4$-manifold $M$ by evaluating it on a simpler stably diffeomorphic $4$-manifold $N$. We exemplify this in Corollary~\ref{cor:explicitformula} where we give an explicit expression for the value of an indecomposable semisimple field theory $Z$ on a connected, simply connected closed oriented $4$-manifold $M$ in terms of the Euler characteristic $\chi(M)$ and signature $\sigma(M)$, and the value of $Z$ on the oriented $4$-manifolds $S^4, S^2\times S^2, \CPt, \overline{\CP}^2$, and the Kummer surface $K3$.

\subsection{$\CPt$-stability, the Gluck twist and emergent fermions}
In Section~\ref{sec:fermion}, we show that the behaviour of a semisimple topological field theory $Z$ on a manifold whose universal cover admits a spin structure strongly depends on the presence of fermions in $Z$. A $4$-dimensional oriented topological field theory is said to have \emph{emergent fermions} if the Gluck twist $\phi\in \mathrm{Diff}(S^2\times S^1)$ acts non-trivially (see Remark~\ref{rem:fermion}). To lift the non-spinnability assumption in Corollary~\ref{cor:main} 3., we prove the following correspondence between the presence of fermions in a $4$-dimensional oriented topological field theory $Z$ and the invariance of $Z$ under $\CPt$-stable diffeomorphisms.

\begin{maintheorem}\label{thm:mainfermion} A semisimple oriented $4$-dimensional topological field theory is invariant under $\CPt$-stable diffeomorphisms\footnote{Two connected compact oriented $4$-bordisms $W,W':M\to N$ are \emph{$\CPt$-stably diffeomorphic} if there are integers $n,m \in \mathbb{Z}_{\geq 0}$ and an orientation preserving diffeomorphism of bordisms between $W\#^n \CPt \#^m \overline{\CP}^2$ and $W'\#^{n} \CPt \#^{m}\overline{CP}^2$.} if and only if the Gluck twist $Z(\phi) \in \End(Z(S^2\times S^1))$ acts as the identity, that is if and only if the theory has no emergent fermions.
\end{maintheorem}

The following corollary is then an immediate consequence of the classification of $\CPt$-stable diffeomorphism classes of closed oriented $4$-manifolds.
\begin{maincor}\label{cor:mainfermion}Let $Z$ be a semisimple oriented $4$-dimensional topological field theory without emergent fermions and let $M$ and $N$ be closed, oriented $4$-manifolds such that there is an orientation-preseving homotopy equivalence $M\to N$. Then, $Z(M)= Z(N)$.

More precisely, for any connected closed oriented $4$-manifold $M$, $Z(M)$ only depends on the Euler characteristic $\chi(M)$, the signature $\sigma(M)$,  the fundamental group $\pi_1(M)$ and the image of the fundamental class $c_*[M]\in H_4(\pi_1(M), \mathbb{Z})$ under a classifying map $c:M\to K(\pi_1(M), 1)$ of the universal cover.
\end{maincor}
Theorem~\ref{thm:mainfermion} and Corollary~\ref{cor:mainfermion} are proven in Section~\ref{sec:fermion}. Comparing Corollaries~\ref{cor:main} and~\ref{cor:mainfermion} raises the following question:
\begin{question}\label{q:teichner} In~\cite[Exm 5.2.4]{teichnerthesis}, Teichner constructs two closed, oriented, homotopy equivalent $4$-manifolds $M$ and $N$ that are not $S^2\times S^2$-stably diffeomorphic. Is there a semisimple topological field theory which distinguishes these $4$-manifolds? By Corollary~\ref{cor:mainfermion} such a field theory necessarily needs to have emergent fermions.
\end{question}

As an application of Theorem~\ref{thm:mainfermion}, we compute the value of an indecomposable semisimple field theory on a connected, simply connected closed oriented $4$-manifold. 
\begin{maincor}\label{cor:explicitfermion}Let $Z$ be an indecomposable semisimple oriented $4$-dimensional topological field theory and let $M$ be a connected, simply connected closed oriented $4$-manifold. \\Depending on whether $M$ admits a spin structure and $Z$ has emergent fermions, $\widetilde{Z}(M):= Z(S^4)^{-1}Z(M)$ can be computed as follows:
\def\hsp{\hphantom{\hspace{1.8cm}}}
\[\arraycolsep=1.4pt\def\arraystretch{2}
\begin{array}{l|c|c} & M\text{ spinnable} & M\text{ non-spinnable}\\ \hline
Z\text{ has fermions}&~~~~~~~~ \widetilde{Z}(K3)^{-\frac{\sigma(M)}{16}} ~\widetilde{Z}(S^2\times S^2)^{\frac{1}{2} \left(\chi(M)-2 + \frac{11}{8} \sigma(M)\right)} ~~~~~~~~&\hsp 0\hsp
\\ \hline
Z \text{ has no fermions }&\multicolumn{2}{c}{\widetilde{Z}(\CPt)^{\frac{1}{2} \left(\chi(M)+\sigma(M)-2\right)}~\widetilde{Z}(\overline{\CP}^2)^{\frac{1}{2}\left(\chi(M)-\sigma(M)-2\right)}}
\end{array}
\]
Moreover, except for the top right entry, all entries in the above table are invertible.
\end{maincor}
Corollary~\ref{cor:explicitfermion} is proven in Section~\ref{sec:fermionapp}.
 
 \subsection{The Crane-Yetter-Kauffman field theory}
Throughout this paper, we use the Crane-Yetter-Kauffman~\cite{CYK} topological field theory $\CY_{\cC}$, defined for an arbitrary ribbon fusion category $\cC$, as our guiding example of an indecomposable semisimple oriented $4$-dimensional topological field theory. Applying Corollary~\ref{cor:main} to this field theory settles in the negative the question~\cite{CYK} of whether $\CY_{\cC}$ for a general (not necessarily modular) ribbon fusion category is sensitive to smooth structure.

The field theory $\CY_{\cC}$ has emergent fermions if and only if $\cC$ contains a `fermion' --- a transparent simple object with non-trivial twist (Example~\ref{exm:CYfermion}). The values of $\CY_{\cC}$ on $S^4, S^2\times S^2, \CPt$ and $\overline{\CP}^2$ encode important algebraic invariants of $\cC$ (see for example~\cite{braided}), namely the \emph{global dimension} of $\cC$ (see Example~\ref{exm:CYone})
\begin{align*}\CY_{\cC}(S^4) &=\cD_{\cC}:= \sum_{X_i \in \mathrm{Irr}(\cC)} \dim(X_i)^2,
\\
\intertext{the \emph{global dimension} of the symmetric center of $\cC$ (see Example~\ref{exm:CYS2})}
\CY_{\cC}(S^4)^{-2} \CY_{\cC}(S^2\times S^2) &=\sum_{X_i \in \mathrm{Irr}(\Zsym(\cC))} \dim(X_i)^2,
\\
\intertext{and the normalized \emph{Gauss sums} (see Example~\ref{exm:CYCPT})}
\CY_{\cC}(S^4)^{-2} \CY_{\cC}(\CPt) &= \cD_{\cC}^{-1} \sum_{X_i \in \mathrm{Irr}(\cC)} \theta_i \dim(X_i)^2 \\ \CY_{\cC}(S^4)^{-2} \CY_{\cC}(\overline{\CP}^2) &= \cD_{\cC}^{-1}\sum_{X_i \in \mathrm{Irr}(\cC)} \theta_i^{-1} \dim(X_i)^2.
\end{align*}
In particular, applying Corollary~\ref{cor:explicitfermion} to $\CY_{\cC}$ leads to a topological proof that a ribbon fusion category over an algebraically closed field of characteristic zero contains a simple transparent object with non-trivial twist if and only if one, or equivalently both, of the Gauss sums are zero (Corollary~\ref{cor:Gausssum}). 

\begin{question} Treating the field theory $\CY_{\cC}$ as an invariant of the ribbon fusion category $\cC$ and following Corollary~\ref{cor:explicitfermion}, one might expect that if $\cC$ contains a transparent simple object with non-trivial twist, the algebraic invariant $\CY_{\cC}(K3)$ should play a similarly important role as the global dimensions $\CY_{\cC}(S^4), \CY_{\cC}(S^2\times S^2)$ and the Gauss sums $\CY_{\cC}(\CPt), \CY_{\cC}(\overline{\CP}^2)$. Since $K3$ admits a handle decomposition without $1$- and $3$-handles~\cite{HarerKasKirby}, this invariant $\CY_{\cC}(K3)$ can be computed by evaluating a certain framed $22$-component link~\cite[Fig 2.15]{HarerKasKirby} labelled by objects of $\cC$ (as described in Example~\ref{exm:CYone}). Can this (much more complicated) invariant be expressed in terms of known invariants of $\cC$?\looseness=-2
\end{question}

\subsection{Related work}
In~\cite{unitary}, Freedman, Kitaev, Nayak, Slingerland, Walker and Wang construct a pairing on formal linear combinations of closed manifolds and investigate its positivity properties. As a consequence, it is shown that a unitary topological field theory cannot distinguish smoothly s-cobordant manifolds. Since any two s-cobordant manifolds are stably diffeomorphic~\cite{quinn83} (but not vice versa) and since any unitary topological field theory is semisimple (Theorem~\ref{thm:unitaryss}, but again not vice versa), our result may be viewed both as a strengthening and a generalization of the $4$-dimensional results of~\cite{unitary}. The efficiency of unitary field theories as invariants of smooth manifolds in other dimensions is studied in~\cite{3dpos, KTpositive}.

To our knowledge, the observation that sensitivity to smooth structure of a 4-dimensional topological field theories requires some form of  nilpotency with respect to connected summing with $S^2 \times S^2$ was first made by Frank Quinn (see e.g.~\cite{quinn}).

Our description of the Crane-Yetter-Kauffman theory $\CY_{\cC}$ mostly follows Barrett and B\"arenz' work on dichromatic invariants~\cite{dichromatic}. In particular, we generalize their formula~\cite[Lem 3.12]{dichromatic} for the value of the dichromatic invariant on simply connected $4$-manifolds to an analogous formula for arbitrary indecomposable semisimple field theories (Corollaries~\ref{cor:explicitformula} and~\ref{cor:explicitfermion}). Their formula in turn is a generalization of a computation in~\cite{CYKCompute} which expresses the $4$-manifold invariant resulting from the Crane-Yetter-Kauffman theory $\CY_{\cC}$ for a modular category $\cC$ (and hence invertible field theory $\CY_{\cC}$, see also~\cite[Sec 1.3]{SchommerPriesTori}) in terms of Euler characteristic and signature. 

Similar to the Crane-Yetter-Kauffman theory, the oriented $4$-manifold invariants in~\cite{dichromatic,Cui,  DR, CuiHopf} are either proven or expected to arise from once-extended topological field theories with values in one of the symmetric monoidal bicategories of the `bestiary of $2$-vector spaces' of~\cite[App A]{III} and should therefore be subject to our results.

\subsection{Outline}

Section~\ref{mainsec:def} concerns the definition and examples of semisimple field theories. After recalling background material in Section~\ref{sec:background}, we define semisimple field theories in Section~\ref{sec:semisimple}. In Sections~\ref{sec:unitary} and~\ref{sec:extended} we prove that both unitary and extended field theories are semisimple.

In Section~\ref{sec:S2S2stability}, we prove Theorem~\ref{thm:main} and investigate its consequences.
 In Section~\ref{sec:main}, after establishing that indecomposable semisimple field theories are multiplicative under connected sums, we combine a certain diffeomorphism of $4$-bordisms (Proposition~\ref{prop:manifolddecomposition}) with a well-known algebraic characterization of semisimple Frobenius algebras (Proposition~\ref{prop:window}), to prove that such field theories do not vanish on $S^2\times S^2$ (Theorem~\ref{thm:S2invertible}). Theorem~\ref{thm:main} then follows from decomposing a semisimple field theory into its components. In Section~\ref{sec:consequencesA}, we prove Corollary~\ref{cor:main} and explicitly compute the $4$-manifold invariant arising from an indecomposable semisimple field theory on simply connected closed $4$-manifolds (Corollary~\ref{cor:explicitformula}). 

In the last Section~\ref{sec:fermion}, we investigate the interplay between $\CPt$-stability and the Gluck twist. Theorem~\ref{thm:mainfermion} is proven in Section~\ref{sec:mainfermion}, again by constructing a certain diffeomorphism of $4$-bordisms (Proposition~\ref{prop:diffeoCP1}) to establish the theorem for indecomposable semisimple field theories (Theorem~\ref{thm:emergentfermions}). In Section~\ref{sec:fermionapp}, we prove Corollaries~\ref{cor:mainfermion} and~\ref{cor:explicitfermion}.

\subsection{Acknowledgements}
I am grateful to Chris Douglas for numerous discussions about field theories and $4$-manifolds, William Olsen for making me aware of the relevance of stable diffeomorphisms to topological field theories, Peter Teichner for help with streamlining the proof of Proposition~\ref{prop:manifolddecomposition}, and to Zhenghan Wang for pointing out the relevance of Corollary~\ref{cor:Gausssum} to the theory of ribbon fusion categories. I am also grateful to the anonymous referees for their comments and suggestions, and for the hospitality and financial support of the Max-Planck Institute for Mathematics where this work was carried out.
\section{Semisimple four-dimensional topological field theories}
\label{mainsec:def}

\subsection{Background}\label{sec:background}
Throughout, we let $k$ be an algebraically closed field and denote the symmetric monoidal category of $k$-vector spaces and linear maps by $\Vect_k$. 

All manifolds appearing in this paper will be smooth and oriented. In Propositions~\ref{prop:manifolddecomposition} and~\ref{prop:diffeoCP1} we use handle diagrams and the Kirby calculus of handle moves to prove that certain closed oriented $4$-manifolds are diffeomorphic. We refer the reader to~\cite{GompfStip} for a thorough introduction to these techniques. Given two closed oriented $(n-1)$-manifolds $M$ and $N$, recall that an \emph{oriented $n$-bordism} $M\to N$ is a compact oriented $n$-manifold $W$ together with an orientation preserving diffeomorphism $i_W:\overline{M} \sqcup N \to \partial W$, where $\overline{M}$ denotes the manifold $M$ with the opposite orientation. An \emph{orientation preserving diffeomorphism of oriented bordisms} $W,W':M\to N$ is an orientation preserving diffeomorphism $f:W\to W'$ such that $f \circ i_W = i_{W'}$. We follow common conventions and surpress the diffeomorphisms $i_W$ from our notation, leaving it to the reader to recover them from context. We let $\Bord_n$ denote the symmetric monoidal category of closed oriented $(n-1)$-manifolds and diffeomorphism classes of $n$-bordisms between them. A careful definition of this category can for example be found in~\cite{Kock}.

Following the Atiyah-Segal axiomatization~\cite{Atiyah,Segal}, an \emph{oriented topological field theory} is a symmetric monoidal functor $\Bord_{n} \to \Vect_k$. Concretely, this amounts to an assignment of a vector space $Z(M)$ to every closed oriented $(n-1)$-manifold $M$ and a linear map $Z(W):Z(M) \to Z(N)$ to every (diffeomorphism class of) oriented $n$-bordism $W:M\to N$, in a way that is compatible with gluing of bordisms and disjoint union.

Recall that a \emph{commutative Frobenius algebra} $(A,m,u, \Delta, \epsilon)$ is a $k$-vector space $A$ equipped with the structure of a commutative algebra $(m:A\otimes A \to A, u:k\to A)$ and a cocommutative coalgebra $(\Delta:A\to A\otimes A , \epsilon:A \to k)$ such that $m$ and $\Delta$ fulfill the following \emph{Frobenius compatibility condition}:
\begin{equation}(\id_A \otimes m)\circ(\Delta \otimes \id_A) =(m\otimes \id_A)\circ (\id_A \otimes \Delta)\label{eq:Frobenius}
\end{equation}
(In any Frobenius algebra, it can be shown that the expression~\eqref{eq:Frobenius} furthermore equals $\Delta \circ m$.)
Commutative Frobenius algebra objects are defined analogously in any symmetric monoidal category. In particular, for $n\geq 2$, the $(n-1)$-sphere $S^{n-1}$ is a commutative Frobenius algebra object in the bordism category $\Bord_n$ with unit $u_{n-1}: \emptyset\to S^{n-1}$ and counit $\epsilon_{n-1}:S^{n-1} \to \emptyset$ given by the $n$-disk $D^n$, and with multiplication $m_{n-1}:S^{n-1} \sqcup S^{n-1} \to S^{n-1}$ and comultiplication $\Delta_{n-1}: S^{n-1} \to S^{n-1} \sqcup S^{n-1}$ given by the `pair of pants bordism' obtained from removing two embedded $n$-disks from an $n$-disk. More generally, since any closed oriented $k$-manifold $M$ ($0\leq k \leq n-1$) induces a symmetric monoidal functor $-\times M: \Bord_{n-k} \to \Bord_{n}$, the manifold $S^{n-k-1} \times M$ admits the structure of a commutative Frobenius algebra object $(m_{n-k-1} \times M, u_{n-k-1}\times M, \Delta_{n-k-1}\times M, \epsilon_{n-k-1}\times M)$.

Besides their well-known role in the classification of two-dimensional oriented topological field theories (see e.g.~\cite{Kock}), commutative Frobenius algebras play important roles in the study of topological field theories in any dimension. Indeed, much of this paper is concerned with the following commutative Frobenius algebras associated to any $4$-dimensional oriented topological field theory.

\begin{definition} Let $Z$ be an oriented $4$-dimensional topological field theory. Its \emph{algebra of local operators} is the commutative Frobenius algebra \[\left(Z(S^3), Z(m_{3}), Z(u_{3}), Z(\Delta_3), Z(\epsilon_3)\right).\] Its \emph{fusion algebra} is the commutative Frobenius algebra \[\left(Z(S^2\times S^1), Z(m_2\times S^1), Z(u_2\times S^1), Z(\Delta_2 \times S^1), Z(\epsilon_2\times S^1)\right).\]
\end{definition}

\begin{remark}\label{rem:physics}
The terminology `algebra of local operators' and `fusion algebra' is inspired by physics.
Physical topological field theories are expected to be \emph{local} or \emph{extended}, also assigning algebraic data to manifolds of higher codimension. Informally, for an $n$-dimensional topological field theory $Z$, the value $Z(S^k)$ (for $0\leq k \leq n-1$) should be thought of as encoding the `collection' (really: an object of some $n-k$-category) of labels of $(n-k-1)$-dimensional strata in $n$-manifolds  (where the sphere $S^k$ is thought of as the linking sphere of that stratum). In particular, for a $4$-dimensional topological field theory, $Z(S^3)$ encodes the `local operators' of the field theory which can be inserted into points of $4$-manifolds. Similarly, for a once-extended theory, $Z(S^2)$ encodes the data labelling $1$-dimensional strata in $4$-manifolds. Equivalently, if we think of our $4$-manifolds as `spacetimes' and of these $1$-dimensional strata as `worldlines' of point particles, $Z(S^2)$ encodes the point particles of the $4$-dimensional field theory. Since the algebra structure on $Z(S^2\times S^1)$ may be thought of as a decategorification, or trace, of the monoidal structure on $Z(S^2)$ induced from inclusions of $3$-disks, it may be thought of as encoding the `fusion of point particles' in the quantum field theory $Z$.
\end{remark}
\begin{example}\label{exm:CYone} Our guiding example throughout this paper is the \emph{Crane-Yetter-Kauffman theory}~\cite{CYK}, an oriented $4$-dimensional topological field theory $\CY_{\cC}:\Bord_{4} \to \Vect_k$ over an algebraically closed field $k$ of characteristic zero, defined for any ribbon fusion category $\cC$ (see~\cite{dichromatic} for a definition of ribbon fusion category).
Our use of the Crane-Yetter-Kauffman theory closely follows~\cite{dichromatic}, where the resulting invariant of closed oriented $4$-manifolds $M$ is expressed\footnote{To extend the Crane-Yetter-Kauffman $4$-manifold invariant from~\cite{dichromatic} to a topological field theory we need to use the normalization denoted $\mathrm{CY}_{\cC}$ in~\cite{dichromatic}, rather than the one used in their main definition and denoted $\widehat{\mathrm{CY}}_{\cC}$. Explicitly, on a closed oriented $4$-manifold $M$, the invariants are related as follows: $\widehat{\mathrm{CY}}_{\cC}(M) = \mathrm{CY}_{\cC}(M)~\cD_{\cC}^{1-\chi(M)}$, where $\cD_{\cC}:= \sum_{X_i \in \mathrm{Irr}(\cC)} \dim(X_i)^2$ is the \emph{global dimension} of the ribbon fusion category $\cC$ and $\chi(M)$ is the Euler characteristic of $M$. See~\cite[Sec 7]{dichromatic} for more details.} in terms of a handle decomposition of $M$ (as a special case amongst a more general family of `dichromatic' invariants).  If $M$ admits a Kirby diagram with a single $0$- and $4$-handle and which is free of $1$- and $3$-handles, the invariant can be computed as follows: Since $\cC$ is a ribbon category, we can evaluate any framed link $L$ with a labelling of each connected component $L_i$ of $L$ by an object $X_i$ of $\cC$ to a scalar $L(X_1,\ldots, X_n)$. To compute $\CY_{\cC}(M)$, we then sum up these scalars over a set of representing simple objects $\mathrm{Irr}(\cC)$ of $\cC$ using appropriate normalization factors:\begin{equation}\label{eq:CraneYetterformula}\CY_{\cC}(M) = \cD_{\cC} \sum_{X_1,\ldots, X_n \in \mathrm{Irr}(\cC)} \left( \prod_{i}\dim(X_i)\right)~L(X_1,\ldots, X_n).\end{equation}
Here, $\dim(X)$ denotes the \emph{quantum dimension} of the object $X$, defined as the evaluation of the $0$-framed unlink labelled by $X$, and $\cD_{\cC}:= \sum_{X \in \mathrm{Irr}(\cC)} \dim(X)^2$ is the \emph{global dimension} of $\cC$.
In the general case, formula~\eqref{eq:CraneYetterformula} has to be adapted slightly to the presence of $1$- and $3$-handles (see~\cite{dichromatic}). 

The algebra of local operators and the fusion algebra of $\CY_{\cC}$ can be understood in terms of the symmetric center of $\cC$: Recall that an object $x$ in a braided monoidal category $\cC$ is \emph{transparent} if it braids trivially with all other objects, that is if  $c_{y,x} \circ c_{x,y} = \id_{x\otimes y}$ for all objects $y$ of $\cC$ where $c_{x,y}:x\otimes y \to y \otimes x$ denotes the braiding natural isomorphism of $\cC$. The \emph{symmetric center} $Z_{\mathrm{sym}}(\cC)$ of $\cC$ is the full monoidal subcategory of $\cC$ on all transparent objects. In particular, $Z_{\mathrm{sym}}(\cC)$ is a symmetric monoidal category and is ribbon if $\cC$ is ribbon (see~\cite[Def 2.41]{dichromatic} for more details). The algebra of local operators $\CY_{\cC}(S^3)$ is the endomorphism-algebra $\Hom_{\Zsym(\cC)}(I, I) \iso k$ of the tensor unit $I$ of $Z_{\mathrm{sym}}(\cC)$ and the fusion algebra $\CY_{\cC}(S^2\times S^1)$ is the $k$-linearized Grothendieck ring\footnote{The \emph{Grothendieck ring} $K_0(\cC)$ of a monoidal semisimple category $\cC$ is as an abelian group freely generated by the isomorphism classes of simple objects of $\cC$ with ring structure induced from the monoidal structure of $\cC$.\looseness=-2} $K_0(\Zsym(\cC))\otimes_{\mathbb{Z}} k$ of $\Zsym(\cC)$ (see~\cite{WalkerNotes} for a proof sketch).

It is expected that the Crane-Yetter-Kauffman theory arises from a fully extended field theory with values in the $4$-category of braided tensor categories~\cite{BJS} and that it is in fact an oriented version of the fully extended framed field theory constructed via the cobordism hypothesis~\cite{baezdolan, lurie} from a braided fusion category in~\cite{BJS}. In particular, the $1$-category `of point particles' $\CY_{\cC}(S^2)$ with its symmetric monoidal structure inherited from embeddings of $3$-disks into $3$-disks is expected to be the symmetric center $\Zsym(\cC)$ of $\cC$.
And indeed, in any once-extended field theory $Z$ (with values in the symmetric monoidal bicategory $\tVect_k$ of additive and idempotent complete $k$-linear categories, see Section~\ref{sec:extended}) both algebras $Z(S^3)$ and $Z(S^2\times S^1)$ are completely determined by the $1$-category $Z(S^2)$ with its induced monoidal structure with monoidal unit $I$; $Z(S^3)$ is the endomorphism algebra $\Hom_{Z(S^2)}(I,I)$ (see e.g. the proof of Theorem~\ref{thm:extended}) while $Z(S^2\times S^1)$ is the $k$-linearized Grothendieck ring of the monoidal category $Z(S^2)$.
\end{example}

In the following, we say that an oriented topological field theory is \emph{zero} if it is zero on all non-empty closed $3$-manifolds and on all non-empty compact $4$-bordisms. 
\begin{prop}\label{prop:nonzero}Let $Z:\Bord_4\to \Vect_k$ be a non-zero oriented $4$-dimensional topological field theory. Then, both its algebra of local operators $Z(S^3)$ and its fusion algebra $Z(S^2\times S^1)$ are non-zero. 
\end{prop}

\begin{proof} 
First note that if $Z(D^4)$ is the zero linear map, it follows by excising and regluing an embedded $4$-disk from the interior of any non-empty $4$-dimensional compact oriented bordism $W$, that $Z(W)=0$. Thus, the topological field theory $Z$ is zero.
In particular, if $Z(S^3)$ is the zero vector space, $Z(D^4)$ is zero and hence $Z$ is zero. 
Similarly, if $Z(S^2\times S^1)$ is zero, by excising and regluing an embedded $S^2\times D^2$ from $D^4$, it again follows that $Z(D^4)=0$ and hence that $Z$ is zero. 
\end{proof}

The \emph{direct sum} $Z_1\oplus Z_2$ of two oriented topological field theories~\cite{durhuus} is defined to be the topological field theory which assigns the vector space $Z_1(M)\oplus Z_2(M)$ to any non-empty connected closed oriented $(n-1)$-manifold $M$ and the tensor product of these spaces to disconnected manifolds. Similarly, to a non-empty connected compact oriented $n$-bordism $W$ it assigns the direct sum of linear maps $Z_1(W)$ and $Z_2(W)$ (interpreted as a linear map between the appropriate tensor products of direct sums) and again extends to non-connected bordisms by taking tensor products.
In particular, the value of $Z_1\oplus Z_2$ on a non-empty connected closed oriented $n$-manifold is simply the sum of the values of $Z_1$ and $Z_2$. We say that a topological field theory is \emph{indecomposable} if it is not isomorphic to a direct sum of non-zero field theories.

Using the fact that for every non-empty connected closed oriented $(n-1)$-manifold $M$, the vector space $Z(M)$ carries a canonical action of the algebra $Z(S^{n-1})$, Sawin~\cite{Sawin} showed that direct sum decompositions of a topological field theory $Z:\Bord_n \to \Vect_k$ are controlled by its algebra of local operators $Z(S^{n-1})$. 

\begin{prop}[{\cite[Thm 1]{Sawin}}]\label{prop:Sawin} The algebra of local operators $Z(S^{n-1})$ is the direct sum of Frobenius algebras $A_1\oplus A_2$ if and only if $Z$ is the direct sum $Z_1\oplus Z_2$ of topological field theories $Z_1$ and $Z_2$ whose algebras of local operators are $A_1$ and $A_2$, respectively. 

In particular, $Z$ is indecomposable if and only if the algebra of local operators of $Z$ is indecomposable as a Frobenius algebra.
\end{prop}

\subsection{Semisimple topological field theories}
\label{sec:semisimple}

\begin{definition}\label{def:semisimple}
An oriented $4$-dimensional topological field theory $Z:\Bord_4 \to \Vect_k$ is \emph{semisimple} if both its algebra of local operators $Z(S^3)$ and its fusion algebra $Z(S^2\times S^1)$ are semisimple.
\end{definition}

\begin{example}\label{exm:invertible} 
Due to their direct amenability to techniques from algebraic topology, the arguably best-understood class of topological field theories are the \emph{invertible field theories}~\cite{FreedShort, SchommerPriesInv}. In our $1$-categorical setting, invertibility of a topological field theory $Z:\Bord_4\to \Vect_k$ amounts to the requirement that all vector spaces $Z(M^3)$ are one-dimensional, and all linear maps $Z(W^4):Z(M^3)\to Z(N^3)$ are invertible. Since any $k$-algebra on a one-dimensional vector space is trivial, every oriented invertible $4$-dimensional topological field theory is automatically semisimple.
\end{example}

Using Proposition~\ref{prop:Sawin}, we observe that every semisimple topological field theory decomposes into a finite direct sum of semisimple field theories with $Z(S^3)\iso k$. 

\begin{prop}\label{prop:decomposess} Every semisimple oriented topological field theory admits a decomposition into a finite direct sum of indecomposable semisimple field theories. A semisimple oriented topological field theory $Z$ is indecomposable if and only if $Z(S^3) \iso k$.
\end{prop}
\begin{proof}
By Artin-Wedderburn, every finite-dimensional semisimple commutative algebra over an algebraically closed field $k$ is a finite direct sum $\oplus_i k$ of copies of the trivial algebra $k$. It therefore follows from Proposition~\ref{prop:Sawin} that $Z$ is indecomposable if and only if $Z(S^3) \iso k$. 

 Suppose that $Z= \bigoplus_i Z_i$ is a semisimple topological field theory, where $Z_i$ are indecomposable topological field theories. We then claim that each component $Z_i$ is itself semisimple. Of course, $Z_i(S^3) \iso k$ is semisimple. By the definition of the direct sum of topological field theories, it follows that the algebra $Z(S^2\times S^1)$ is a direct sum of the algebras $Z_i(S^2\times S^1)$. The claim then follows since every component in a direct sum decomposition of a semisimple algebra is again semisimple. 
 \end{proof}

\subsection{Unitary topological field theories are semisimple}\label{sec:unitary}
In this section, we work over the field $k=\mathbb{C}$ and prove that every unitary topological field theory is semisimple.

For a bordism $W:M \to N$, we let $\overline{W}:N\to M$ denote the bordism with opposite orientation (and hence source and target interchanged). A \emph{unitary topological field theory} is a symmetric monoidal functor $\Bord_n\to \Hilb$ into the symmetric monoidal category of Hilbert spaces and linear maps such that $Z(\overline{W}) = Z(W)^\dagger: Z(N) \to Z(M)$ is the adjoint linear map of $Z(W):Z(M)\to Z(N)$.
In other words, both symmetric monoidal categories $\Bord_n$ and $\Hilb$ admit a \emph{dagger structure} (also known as $*$-structure) and a unitary topological field theory is required to preserve that structure.

\begin{theorem}\label{thm:unitaryss} Any unitary topological field theory is semisimple.
\end{theorem}
\begin{proof} For a unitary topological field theory, $Z(S^3)$ and $Z(S^2\times S^1)$ are commutative $\dagger$-Frobenius algebras in $\Hilb$, that is commutative Frobenius algebras such that the comultiplication $\Delta$ is the adjoint linear map of the multiplication: $\Delta= m^\dagger$. The theorem then follows from the fact that the underlying finite-dimensional $\mathbb{C}$-algebra of every $\dagger$-Frobenius algebra in $\Hilb$ admits the structure of a finite-dimensional $C^*$-algebra~\cite[Cor 4.3]{ONB} and is therefore semisimple.
\end{proof}

\subsection{Once-extended $k$-linear topological field theories are semisimple}\label{sec:extended}

Most existing topological field theories --- and in particular the ones motivated by physics --- are either proven or believed to be \emph{extended}, meaning that they also assign algebraic data to manifolds of higher codimension and allow gluing not only along boundaries but also along higher codimensional corners. Here, we show that any topological field theory that is `once-extended' in the sense that it also assigns $k$-linear categories to closed oriented $2$-manifolds is automatically semisimple. 

We follow~\cite{SchommerPries} and let $\Bord_{4,3,2}$ denote the symmetric monoidal bicategory of \emph{once-extended} oriented bordism. Roughly speaking, its objects are closed oriented $2$-manifolds, its $1$-morphisms are compact oriented $3$-dimensional bordism and its $2$-morphisms are diffeomorphism classes of compact oriented $4$-dimensional bordisms with corners. We refer to~\cite{SchommerPries} for a precise definition of this symmetric monoidal bicategory. 

In the following, a \emph{once-extended $k$-linear oriented $4$-dimensional topological field theory} is a symmetric monoidal $2$-functor $\Bord_{4,3,2}\to \tVect_k$, where $\tVect_k$ is the symmetric monoidal bicategory of additive and idempotent complete $k$-linear categories, $k$-linear functors and natural transformations. 
In fact, there are many possible symmetric monoidal bicategories $T$ which serve as a potential target `extending' $\Vect_k$ (in the sense that $\Hom_{T}(I_{T},I_{T}) \iso \Vect_k$). In~\cite[App A]{III} various other natural candidates for such bicategories of `2-vector spaces' are discussed, including the bicategory of $k$-algebras, $k$-bimodules and bimodule maps. By restricting to closed oriented $3$-manifolds and bordisms between them, any once extended $k$-linear topological field theory $Z: \Bord_{4,3,2} \to T$ induces an ordinary $4$-dimensional topological field theory $\Omega Z: \Bord_{4} \to \Vect_k$ in the sense of Section~\ref{sec:background}.

\begin{theorem}\label{thm:extended} Let $Z:\Bord_{4,3,2} \to T$ be a once extended $k$-linear oriented $4$-dimensional topological field theory, where $T$ is $\tVect_k$ or any of the other symmetric monoidal bicategories of~\cite[App A]{III}. Then, $\Omega Z:\Bord_4\to \Vect_k$ is semisimple. 
\end{theorem}
\begin{proof} 
Based on an observation of Tillmann~\cite{tillmann}, it is shown in~\cite[Thm A.22]{III} that the symmetric monoidal bicategory $\tVect_k^{\mathrm{f.d.}}$ of finite semisimple $k$-linear categories, $k$-linear functors and natural transformations is equivalent to the fully dualizable subcategory of any of the bicategories $T$. Therefore, any topological field theory $\Bord_{4,3,2} \to T$ factors through $\tVect_k^{\mathrm{f.d.}}$ and we may henceforth assume that $T=\tVect_k^{\mathrm{f.d.}}$.

The cancellation of $0$- and $1$-handles gives rise to an adjunction\footnote{Recall that a pair of $1$-morphisms $f:a\to b$ and $g:b\to a$ in a bicategory form an \emph{adjunction} if there are $2$-morphisms $\eta: \id_a \To g \circ f$ and $\epsilon: f \circ g \To \id_b$, called the unit and counit of the adjunction, such that (omitting coherence isomorphisms for better readability) $(\epsilon \circ \id_{f}) \cdot (\id_{f} \circ \eta) = \id_{f}$ and $(\id_g \circ \epsilon) \cdot (\eta \circ \id_g) = \id_g$.  A $1$-morphism $f:a\to b$ which is part of an adjunction as above is said to be a \emph{left adjoint}. Given a left adjoint $1$-morphism $f:a\to b$, its \emph{right adjoint} $g:b\to a$ and the unit $\eta$ and counit $\epsilon$ of the adjunction are uniquely determined up to isomorphism. } between the $1$-morphisms $D^3: \emptyset \to S^2$ and $D^3: S^2 \to \emptyset$ in the bordism bicategory $\Bord_{4,3,2}$ with the following unit and counit:
\[\nonumber
\begin{tikzcd}
& S^2\arrow[dr, bend left =20,"D^3" ]&\\
\emptyset\arrow[rr, "\emptyset"'] \arrow[ur, bend left =20, "D^3"] &\arrow[u,Rightarrow, "D^4"']& \emptyset
\end{tikzcd}
\hspace{2cm}
\begin{tikzcd}
S^2 \arrow[rr, "D^1 \times S^2"]\arrow[dr, bend right=20, "D^3"']&\arrow[d, Leftarrow, "D^1\times D^3"]&S^2\\
& \emptyset \arrow[ur, bend right=20, "D^3"'] &
\end{tikzcd}
\]

Omitting coherence isomorphisms from the notation, any adjunction $(f:a \to b,g:b\to a, \eta:\id_a\To g\circ f, \epsilon:f\circ g \To \id_b)$ in a bicategory $\cB$ gives rise to an algebra object $\left(g\circ f, g\circ f\circ g \circ f \To[g\circ \epsilon \circ f] g\circ f, \id_a\To[\eta] g\circ f\right)$ in the monoidal category $\Hom_{\cB}(a,a)$. Since the right adjoint $g$, the unit $\eta$ and the counit $\epsilon$ are uniquely determined up to isomorphism by $f$, it follows that this algebra is also uniquely determined by $f$ up to algebra isomorphism.

Decomposing $S^3$ as $\emptyset\to[D^3] S^2 \to[D^3] \emptyset$ and observing that the `pair of pants bordism' $m_3: S^3 \sqcup S^3 \to S^3$ of Section~\ref{sec:background} corresponds to a $1$-handle attachment, it follows that the algebra structure $(S^3, m_3, u_3)$ of Section~\ref{sec:background} indeed arises in this way from the adjunction between $D^3:\emptyset \to S^2$ and $D^3: S^2\to \emptyset$. 

Up to isomorphism, the algebra $\left(Z(S^3), Z(m_3), Z(u_3)\right)$ is therefore uniquely determined by the linear functor $Z(D^3): Z(\emptyset) \to Z(S^2)$. Since the monoidal unit of $\tVect_k^{\mathrm{f.d.}}$ is the finite-semisimple category $\Vect_k^{\mathrm{f.d.}}$ of finite-dimensional vector spaces, and since every linear functor from $\Vect_k^{\mathrm{f.d.}}$ into a finite semisimple category is completely determined by where it sends the one-dimensional vector space $k$, we will henceforth tacitcly identify the category of linear functors $\Vect_k^{\mathrm{f.d.}}\to \cC$ with $\cC$ itself. It follows that $Z(D^3):Z(\emptyset)\iso \Vect_k^{\mathrm{f.d.}} \to Z(S^2)$ singles out an object\footnote{Of course, this object $I$ is the tensor unit of the monoidal structure on $\cC$ induced from embeddings of $3$-disks. Even though we will henceforth denote this object by $I$, this observation is not necessary for our proof of Theorem~\ref{thm:extended}.} $I$ of the category $Z(S^2)$. A right adjoint of this functor is the functor $\Hom_{Z(S^2)}(I, -): Z(S^2) \to \Vect_k^{\mathrm{f.d.}}$ and the resulting algebra structure on the composite $\Hom_{Z(S^2)}(I,I)$ (understood as an object of $\Vect_k^{\mathrm{f.d.}}$) is simply the usual algebra structure induced by composition in the category $Z(S^2)$. Since $Z(S^2)$ is a semisimple category, it follows that the endomorphism algebra $\Hom_{Z(S^2)}(I,I)$ is semisimple.

Applying the same argument to the `dimensionally reduced' topological field theory $Z(-\times S^1): \Bord_{3,2,1} \to \tVect$ implies that $Z(S^2\times S^1)$ is the endomorphism algebra $\Hom_{Z(S^1\times S^1)}(I,I)$ of some object $I$ in the semisimple category $Z(S^1\times S^1)$ and is therefore also semisimple. 
\end{proof}

\begin{example} As discussed in Example~\ref{exm:CYone}, the Crane-Yetter-Kauffman theory is expected to be fully extended taking values in the $4$-category of braided tensor categories~\cite{BJS, WalkerNotes}. In particular, Theorem~\ref{thm:extended} should imply that $\CY_{\cC}$ is semisimple. Alternatively, one can directly show that the algebra of local operators $\Hom_{\Zsym(\cC)}(I,I) \iso k$ and the fusion algebra $K_0(\Zsym(\cC))\otimes_{\mathbb{Z}} k$ are semisimple\footnote{In~\cite[Cor 3.7.7]{EGNO}, it is shown that for any fusion ring $A$~\cite[Def 3.1.7]{EGNO}, the algebra $A_{\mathbb{C}}:=A\otimes_{\mathbb{Z}}\mathbb{C}$ is semisimple. To extend this to arbitrary fields of characteristic zero, note that any fusion ring $A$ admits a canonical Frobenius algebra structure in the category of free abelian groups (using the notation of~\cite[Sec 3.1]{EGNO}, comultiplication and counit are given by $\Delta(b_i):= \sum_j b_j\otimes b_j^* b_i $ and $\tau(1) := 1$, $\tau(b_i)=0$ for $b_i \neq 1$). In particular, by Proposition~\ref{prop:window}, $A_k=A\otimes_{\mathbb{Z}}k$ is semisimple if and only if the endomorphism (of free abelian groups) $f:=m\circ \Delta:A \to A$ is invertible over $k$. But since $f$ is invertible over $\mathbb{C}$ and since $k$ is of characteristic zero (as assumed for the Crane-Yetter theory), it follows that $f$ is invertible over $k$ and hence that $A_k$ is semisimple.}~\cite[Cor 3.7.7]{EGNO}.\looseness=-2
\end{example}

\section{Semisimple topological field theories are $S^2\times S^2$-stable}\label{sec:S2S2stability}
Let $X$ be a connected, closed, oriented $n$-manifold. Two connected compact oriented $n$-bordisms $M, N:A\to B$ are \emph{$X$-stably diffeomorphic} if there are natural numbers $k_{+}, k_{-} \geq 0$ and an orientation-preserving diffeomorphism of oriented bordisms 
\[M \#^{k_+} X \#^{k_-} \overline{X} \iso N \#^{k_+}X \#^{k_-}\overline{X}.
\]
Here, and throughout this paper, connected sums are taken in the interior of bordisms.

\begin{definition} An $n$-dimensional oriented topological field theory is \emph{$X$-stable} if $Z(M) = Z(N)$ for $X$-stably diffeomorphic connected compact oriented $n$-bordisms $M$ and $N$.
\end{definition}

In this paper, we will be concerned with $X=S^2\times S^2$ and $X=\CPt$ (note that $S^2\times S^2$ admits an orientation reversing diffeomorphism, whereas $\CPt$ does not).

Most results in this paper ultimately follow from the following well-known and straight-forward observation about field theories with trivial algebra of local operators. 
\begin{prop}\label{prop:multiplicativityconnectsum} Let $Z$ be an oriented $n$-dimensional topological field theory with $Z(S^{n-1}) \iso k$. Then, $Z(S^n)$ is invertible and $Z$ is multiplicative under connected sums: For a connected closed oriented $n$-manifold $M$ and a connected oriented $n$-bordism $N:A\to B$, the following holds, where the connected sum is taken in the interior of $N$:
\[ Z(M\# N ) = Z(S^n)^{-1} Z(M) Z(N)
\]
\end{prop}
\begin{proof} Since $Z(D^n): Z(\emptyset) \to Z(S^{n-1})$ and $Z(D^n): Z(S^{n-1}) \to Z(\emptyset)$ are the unit and counit of a Frobenius algebra on $k$, it follows that they and hence also their composite $Z(S^n)$ is invertible.
Note that Proposition~\ref{prop:multiplicativityconnectsum} follows from the special case $A=\emptyset$ by precomposing $N:A\to B$ with the bordism $A\times D^1:\emptyset \to A\sqcup \overline{A}$. For $A=\emptyset$, the proposition follows from the following decomposition: 
\[Z(M\# N) = Z(\emptyset)\to[Z(M \setminus D^n)] Z(S^{n-1}) \to [Z(N \setminus D^n)] Z(B)  
\]
\[=Z(\emptyset) \to [\!\!\! Z(M\setminus D^n)\!\!\!] Z(S^{n-1}) \to[\!\!\!Z(D^n)\!\!\!] Z(\emptyset)\to[\!\!\!Z(D^n)^{-1}\!\!\!\!\!\!] Z(S^{n-1})\to[\!\!\!Z(D^n)^{-1}\!\!\!\!\!\!] Z(\emptyset) \to [\!\!\!Z(D^n)\!\!\!] Z(S^{n-1}) \to [\!\!\!Z(N\setminus D^n)\!\!\!]Z(B)
\]
\[ = Z(\emptyset)\to[Z(M)] Z(\emptyset) \to[Z(S^{n})^{-1}] Z(\emptyset) \to [Z(N)] Z(B)\qedhere
\]
\end{proof}
\subsection{Proving $S^2\times S^2$-stability}\label{sec:main}
Combining Propositions~\ref{prop:decomposess} and~\ref{prop:multiplicativityconnectsum}, proving $S^2\times S^2$-stability of semisimple topological field theories is equivalent to proving invertibility of $Z(S^2\times S^2)$ for indecomposable semisimple topological field theories. In the following section, we achieve this by combining a certain diffeomorphism of $4$-bordisms with the following well-known algebraic characterization of semisimple Frobenius algebras.

\begin{prop}\label{prop:window}A commutative Frobenius algebra $(A,m,u, \Delta, \epsilon)$ is semisimple if and only if the `window endomorphism' $m\circ \Delta:A \to A$ is invertible.
\end{prop}
\begin{proof} Recall that a $k$-algebra $(A,m:A\otimes A \to A, u:k \to A)$ is \emph{separable} if there exists a \emph{separating morphism} $\delta:A \to A\otimes A$ such that $m\circ \delta = \id_A$ and such that $m$ and $\delta$ fulfill the Frobenius condition~\eqref{eq:Frobenius}. Over an algebraically closed field $k$, the notions of separability and semisimplicity are equivalent\footnote{This is more generally true over perfect fields. For more general fields, separability is a strengthening of semisimplicity and is equivalent to the statement that $A\otimes_k K$ is semisimple for every field extension $K$ of $k$. \looseness=-2}~\cite{separable}. 

Let $\delta: A\to A\otimes A$ be a separating morphism for $(A,m,u)$. Using the Frobenius condition~\eqref{eq:Frobenius} and the fact that $\delta$ is a separating morphism, it can then directly be shown that $x:= (\id_A \otimes \epsilon)\circ \delta = (\epsilon \otimes \id_A)\circ \delta: A \to A$ is inverse to $m\circ \Delta$.

Conversely, if $m\circ \Delta:A\to A$ is invertible with inverse $f:A\to A$, then $\Delta \circ f$ is a separating morphism.
\end{proof}

In $\Bord_4$, the `window endomorphism' of the commutative Frobenius algebra object $(S^2\times S^1, m_2 \times S^1, u_2\times S^1, \Delta_2 \times S^1, \epsilon_2 \times S^1)$ is the composite bordism
\[W_{S^2\times S^1}:= S^2 \times S^1 \to[ \Delta_2 \times S^1] S^2\times S^1 \sqcup S^2 \times S^1 \to [m_2 \times S^1] S^2 \times S^1.
\]
This endomorphism has `eigenvector' $S^2\times D^2$, as witnessed by the following diffeomorphism.%
\begin{prop}\label{prop:manifolddecomposition} There is an orientation preserving diffeomorphism of the following oriented bordisms $\emptyset \to S^2\times S^1$:
\begin{equation}\label{eq:firstdiffeo}  W_{S^2\times S^1} \cup_{S^2\times S^1} (S^2 \times D^2)  \iso (S^2\times D^2) \# (S^3 \times S^1)  \#(S^2 \times S^2)  \end{equation}
Here, the connected sum $\#$ is taken in the interior of $S^2\times D^2$.
\end{prop}
\begin{proof}
Let $L$ and $R$ denote the bordisms $\emptyset \to S^2\times S^1$ on the left and right of~\eqref{eq:firstdiffeo}, respectively. We construct a diffeomorphism of bordisms between $L$ and $R$ by composing them with the bordism $D^3\times S^1:S^2\times S^1\to \emptyset$ and constructing a diffeomorphism of pairs
\begin{equation}\label{eq:diffeopairs}\left( \widetilde{L}:=(D^3\times S^1) \cup_{S^2\times S^1} L , D^3\times S^1 \right) \to \left(\widetilde{R}:=(D^3\times S^1) \cup_{S^2\times S^1} R, D^3\times S^1 \right).
\end{equation}%

Using the standard diffeomorphism $(D^3\times S^1)\cup_{S^2\times S^1} (S^2\times D^2)\iso S^4$, it follows that the closed oriented $4$-manifold $\widetilde{R}$ is diffeomorphic to $ (S^3\times S^1)\#(S^2\times S^2) $. Moreover, the embedded circle $D^3\times S^1\hookrightarrow \widetilde{R}$ is null-isotopic.

Observe that the composite bordism $S^2\to[\Delta_2] S^2\sqcup S^2\to[m_2] S^2$ is diffeomorphic to the bordism $(S^2\times S^1)_2: S^2\to S^2$ obtained from removing two $3$-disks from the closed oriented $3$-manifold $S^2\times S^1$. Therefore, $(D^3\times S^1)\cup_{S^2\times S^1} W_{S^2\times S^1}$ is diffeomorphic to the bordism $(S^2\times S^1)_1 \times S^1: S^2\times S^1\to \emptyset$, where $(S^2\times S^1)_1$ is obtained from removing a single $3$-disk from $S^2\times S^1$. In particular, the embedded circle \[D^3\times S^1\hookrightarrow \widetilde{L}= (S^2\times S^1)_1 \times S^1 \cup_{S^2\times S^1} (S^2\times D^2)\] is null-isotopic. Since $D^3\times S^1\hookrightarrow \widetilde{R}$ is also null-isotopic, any orientation-preserving diffeomorphism of the closed oriented $4$-manifolds $\widetilde{L} \to \widetilde{R}$ may be isotoped to a diffeomorphism of pairs~\eqref{eq:diffeopairs}. 

The closed oriented $4$-manifold $\widetilde{L} \iso (S^2\times S^1)_1\times S^1 \cup_{S^2\times S^1} (S^2\times D^2)$ is obtained from $S^2\times S^1\times S^1$ by performing surgery --- that is, replacing an embedded $D^3\times S^1$ by an $S^2\times D^2$ --- on the last $S^1$ of $S^2\times S^1 \times S^1$. In Akbulut's dotted circle notation for $1$-handles~\cite[Sec 5.4]{GompfStip}, the standard handle diagram for $S^2\times S^1 \times S^1$ is a Borromean link with two dotted and one $0$-framed component, together with an additional $0$-framed unknot around the meridian of the $0$-framed component:
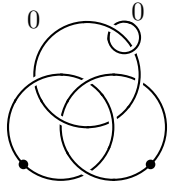
\[\begin{tz}[scale=0.7]
\path  (0.5,2) arc (90:45:1) node (p){};
\begin{knot}[
    clip width=4,
    ignore endpoint intersections=false
    ]
    \strand[thick] (0,0) circle (1.0cm);
    \strand [thick ] (1,0) circle (1.0cm);
    \strand  [thick] (0.5,1) circle (1.0cm);
    \strand[thick] (p) circle (0.3cm);
    \flipcrossings{1, 2, 5, 6,8}
\end{knot}
\path  (1,0) arc (0:225:1) node[pos=1,dot]{};
\path (0,0) arc (180:315:1) node[pos=1, dot]{};
\path (1.5,1) arc (0:135:1) node[above left] {$0$};
\path (p) arc (0:30:0.3) node[above right] {$0$};
\end{tz}
\]
Performing surgery on the last circle of $S^2\times S^1\times S^1$ corresponds to replacing one of the dotted circles with a $0$-framed circle. Sliding this new $0$-framed circle (twice) over the small $0$-framed meridian circle disentangles it from the rest of the link and allows it to be cancelled against a $3$-handle:
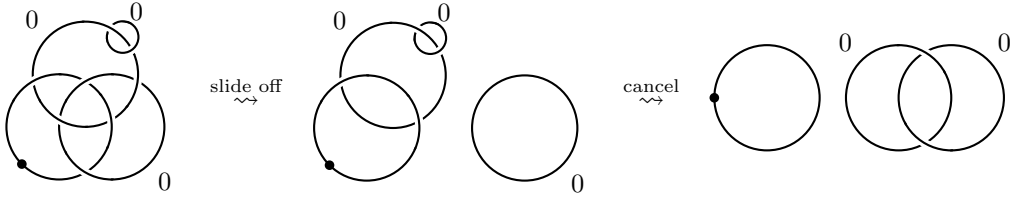
\[\begin{tz}[scale=0.7]
\path  (0.5,2) arc (90:45:1) node (p){};
\begin{knot}[
    clip width=4,
    ignore endpoint intersections=false
    ]
    \strand[thick] (0,0) circle (1.0cm);
    \strand [thick ] (1,0) circle (1.0cm);
    \strand  [thick] (0.5,1) circle (1.0cm);
    \strand[thick] (p) circle (0.3cm);
    \flipcrossings{1, 2, 5, 6,8}
\end{knot}
\path  (1,0) arc (0:225:1) node[pos=1,dot]{};
\path (0,0) arc (180:315:1) node[pos=1, below right]{0};
\path (1.5,1) arc (0:135:1) node[above left] {$0$};
\path (p) arc (0:30:0.3) node[above right] {$0$};
\end{tz}
\handlegap\overset{\text{slide off}}{\rightsquigarrow}\handlegap
\begin{tz}[scale=0.7]
\path  (0.5,2) arc (90:45:1) node (p){};
\begin{knot}[
    clip width=4,
    ignore endpoint intersections=false
    ]
    \strand[thick] (0,0) circle (1.0cm);
    \strand [thick ] (3,0) circle (1.0cm);
    \strand  [thick] (0.5,1) circle (1.0cm);
    \strand[thick] (p) circle (0.3cm);
    \flipcrossings{ 4}
\end{knot}
\path  (1,0) arc (0:225:1) node[pos=1,dot]{};
\path (2,0) arc (180:315:1) node[pos=1, below right]{0};
\path (1.5,1) arc (0:135:1) node[above left] {$0$};
\path (p) arc (0:30:0.3) node[above right] {$0$};
\end{tz}
\handlegap\overset{\text{cancel}}{\rightsquigarrow}\handlegap
\begin{tz}[scale=0.7]
\begin{knot}[
    clip width=4,
    ignore endpoint intersections=false
    ]
    \strand[thick] (-2.5,0) circle (1cm);
    \strand[thick] (0,0) circle (1.0cm);
    \strand [thick ] (1,0) circle (1.0cm);
    \flipcrossings{2}
\end{knot}
\path  (-1.5,0) arc (0:180:1) node[pos=1,dot]{};
\path (1,0) arc (0:135:1) node[pos=1, above left]{0};
\path (2,0) arc (0:45:1) node[above right] {$0$};
\end{tz}
\]
The resulting handle diagram is precisely a diagram for $\widetilde{R}\iso (S^3\times S^1) \#(S^2\times S^2) $.
\end{proof}

Combining the diffeomorphism of Proposition~\ref{prop:manifolddecomposition} with Proposition~\ref{prop:window} results in the following theorem.
\begin{theorem} \label{thm:S2invertible}Let $Z$ be an indecomposable semisimple oriented $4$-dimensional topological field theory. Then, $Z(S^2\times S^2)$ is invertible. 
\end{theorem}
\begin{proof} 
By definition, the commutative Frobenius algebra $Z(S^2\times S^1)$ is semisimple. Hence, by Proposition~\ref{prop:window}, the window endomorphism $Z(w_{S^2\times S^1}): Z(S^2\times S^1) \to Z(S^2\times S^1)$ is invertible. Applying $Z$ to the diffeomorphism of Theorem~\ref{prop:manifolddecomposition} and using multiplicativity under connected sums (Proposition~\ref{prop:multiplicativityconnectsum}), we find:
\[Z(W_{S^2\times S^1}) \circ Z(S^2\times D^2) =  Z(S^2\times D^2) ~\left(Z(S^4)^{-2}  Z(S^3\times S^1) Z(S^2\times S^2)\right).
\]
Since $Z(S^2\times D^2):k \to Z(S^2\times S^1)$ is the unit of the (non-zero by Proposition~\ref{prop:nonzero}) algebra $(Z(S^2\times S^1), Z(S^2\times m_1), Z(S^2\times u_1))$ and hence a non-zero vector in $Z(S^2\times S^1)$, it follows that the scalar $Z(S^4)^{-2}Z(S^3\times S^1) Z(S^2\times S^2)\in k$ is an eigenvalue of the invertible endomorphism $Z(W_{S^2\times S^1})$ and is therefore itself invertible.
\end{proof}

\begin{example}\label{exm:CYS2} Using~\eqref{eq:CraneYetterformula} and the standard handle diagram of $S^2\times S^2$ with two $2$-handles attached along a $0$-framed Hopf link, the value of $\CY_{\cC}(S^2\times S^2)$ can be explicitly computed as a product of the global dimensions of $\cC$ and of the symmetric center of $\cC$ (see~\cite[Sec 6.1]{dichromatic}):\looseness=-2
\[\CY_{\cC}(S^2\times S^2) = \cD_{\cC}^2  \sum_{X \in Z_{sym}(\cC)} \dim(X)^2
\]
In particular, Theorem~\ref{thm:S2invertible} may be understood as a geometric analogue of the invertibility of the global dimension of braided fusion categories~\cite{braided}.
\end{example}
Our main Theorem~\ref{thm:main} is a direct corollary of Theorem~\ref{thm:S2invertible}.
\begin{proof}[Proof of Theorem~\ref{thm:main}]
By Proposition~\ref{prop:decomposess}, we can assume that $Z$ is indecomposable and in particular that $Z(S^3)\iso k$. Let $n\in \mathbb{Z}_{\geq 0}$ be such that $M\#^n(S^2\times S^2)\iso N\#^n(S^2\times S^2)$. Multiplicativity of indecomposable semisimple field theories (Proposition~\ref{prop:decomposess}) implies that
\[Z(M) ~Z(S^4)^{-n} Z(S^2\times S^2)^n = Z(N) ~Z(S^4)^{-n}Z(S^2\times S^2)^n.
\]
Since $Z(S^2\times S^2)$ is invertible (Theorem~\ref{thm:S2invertible}), it follows that $Z(M)= Z(N)$.
\end{proof}

\subsection{Consequences of Theorem~\ref{thm:main}}\label{sec:consequencesA}
Stable diffeomorphism is a very well-studied equivalence relation on the set of closed oriented $4$-manifolds. A classical theorem of Wall~\cite{Wall} shows that two connected, simply connected closed oriented $4$-manifolds are $S^2\times S^2$-stably diffeomorphic if and only if they have isomorphic intersection forms. More generally, Kreck~\cite{Kreck} showed that the stable diffeomorphism class of a closed oriented $4$-manifold is completely determined by its oriented $1$-type and the bordism class of the manifold in an appropriately structured bordism group. This essentially reduces the classification of stable diffeomorphism classes to a bordism problem. Kreck's classification is particularly simple in the non-spinnable case: Indeed, the $S^2\times S^2$-stable diffeomorphism class of any closed oriented $4$-manifold $M$ whose universal cover does not admit a spin structure is completely determined by its Euler characteristic $\chi(M)$, signature $\sigma(M)$, fundamental group $\pi_1(M)$ and the image of the fundamental class $c_*[M] \in H_4(\pi_1(M), \mathbb{Z})$ under a classifying map of the universal cover $c:M\to K(\pi_1(M), 1)$.

Using these results and a theorem of Gompf~\cite{Gompf}, we obtain a proof of Corollary~\ref{cor:main}.

\begin{proof}[Proof of Corollary~\ref{cor:main}]
The first statement follows from a theorem of Gompf~\cite{Gompf} which shows that two homeomorphic closed oriented $4$-manifolds are $S^2\times S^2$-stably diffeomorphic.
The second statement follows from Wall's theorem~\cite{Wall}.
The last statement is a direct consequence of Kreck's classification~\cite{Kreck} of $S^2\times S^2$-stable diffeomorphism classes of closed oriented $4$-manifolds with universal covers which do not admit spin structures.
\end{proof}
\begin{remark}Alternatively, Corollary~\ref{cor:main} 2. also follows from Corollary~\ref{cor:main} 1. and Freedman's celebrated classification of simply connected topological $4$-manifolds~\cite{Freedman} which implies that two simply connected closed smooth $4$-manifolds are homeomorphic if and only if they have the same homotopy type.
\end{remark}

\begin{remark}
In Corollary~\ref{cor:main}, the assumptions on simply-connectedness or non-spinnability of universal covers are due to the relatively simple classification of stable diffeomorphism types in these situations. For the classification of stable diffeomorphism types of more general manifolds, see e.g.~\cite{teichnerthesis, davis} or Question~\ref{q:teichner}.
\end{remark}

Using Theorem~\ref{thm:main}, we can evaluate an indecomposable semisimple theory $Z$ on a manifold $M$ by evaluating it on a potentially much simpler `reference manifold' in the same $S^2\times S^2$-stable diffeomorphism class. For example, to evaluate $Z$ on a simply connected closed oriented $4$-manifold $M$, we only need to know the Euler characteristic and signature of $M$ and the value of $Z$ on the 
manifolds $S^4$, $S^2\times S^2$, $\CPt$, $\overline{\CP}^2$ and the Kummer surface $K3$.

\begin{corollary}\label{cor:explicitformula} Let $Z$ be an indecomposable semisimple oriented $4$-dimensional topological field theory and let $M$ be a connected, simply connected closed oriented $4$-manifold. If $M$ does not admit a spin structure, then 
\begin{equation}\label{eq:simplyconnectednonspin}\widetilde{Z}(M) = \widetilde{Z}(\CPt)^{\frac{1}{2} \left(\chi(M)+\sigma(M)-2\right)}~\widetilde{Z}(\overline{\CP}^2)^{\frac{1}{2}\left(\chi(M)-\sigma(M)-2\right)}.
\end{equation}
If $M$ admits a spin structure, then 
\begin{equation}\label{eq:simplyconnectedspin}\widetilde{Z}(M)= \widetilde{Z}(K3)^{-\frac{\sigma(M)}{16}} ~\widetilde{Z}(S^2\times S^2)^{\frac{1}{2} \left(\chi(M)-2 + \frac{11}{8} \sigma(M)\right)}.
\end{equation}
Here, for a closed oriented $4$-manifold $M$, we use the notation $\widetilde{Z}(M):= Z(S^4)^{-1} Z(M)$. 
\end{corollary}
\begin{proof}By Wall's theorem~\cite{Wall}, two connected, simply connected closed oriented $4$-manifolds are $S^2\times S^2$-stably diffeomorphic if and only if they have isomorphic intersection forms. As a consequence of Donaldson's theorem~\cite{donaldson} this is equivalent to requiring their Euler characteristic, signature and parity to agree. Since the parity of the intersection form of a simply connected smooth $4$-manifold is even if and only if the manifold admits a spin structure, the classification of $S^2\times S^2$-stable diffeomorphism classes of simply connected $4$-manifolds splits into a spin and non-spin case.

In the non-spin case, we note that $\CPt$ is a simply connected non-spin manifold with Euler characteristic $3$ and signature $1$. Since $M$ is a simply connected $4$-manifold which does not admit a spin structure, it cannot be a homology sphere and hence has $b_2^+(M)+b_2^-(M)= b_2(M) \geq 1$. Therefore, the simply connected $4$-manifold $N:=\#^{b_2^+(M)}\CPt \#^{b_2^-(M)} \overline{\CPt}$ has at least one factor of $\CPt$ or $\overline{\CP}^2$ and hence does not admit a spin structure. Since its Euler characteristic is $\chi(M)$ and its signature is $\sigma(M)$, it follows that $N$ is $S^2\times S^2$-stably diffeomorphic to $M$. Formula~\eqref{eq:simplyconnectednonspin} follows from multiplicativity of $Z$ up to $Z(S^4)$ factors.

For the spin case, we note that the $K3$-surface is a simply connected spin manifold with Euler characteristic $24$ and signature $-16$. By Rohlin's theorem~\cite{Rohlin}, the signature $\sigma(M)$ of a closed spin manifold $M$ is divisible by 16 and since $M$ is simply connected and hence $\chi \equiv \sigma~ (\mathrm{mod} ~2)$, the Euler characteristic $\chi(M)$ is even.
Assuming that $M$ has negative signature $\sigma(M)= -16 s \leq 0$, it follows that the following two simply connected spin manifolds have the same Euler characteristic (namely $\chi(M)+2a=2+22s+2b$) and signature (namely $\sigma(M) = -16s$) and are therefore $S^2\times S^2$-stably diffeomorphic:
\begin{equation}\label{eq:spincaseproof}M \#^{a} (S^2\times S^2) \hspace{2cm} \#^{s} K3 \#^{b}(S^2\times S^2)
\end{equation}
Here, $a$ and $b$ are non-negative integers such that $b-a =\frac{1}{2}\left( \chi(M)-2-22s\right)$. Since $\chi(M)$ is even, it is always possible to choose such integers\footnote{In fact, it is not known whether one can always choose $a=0$. This is equivalent to the simply connected case of the celebrated `$\frac{11}{8}$-conjecture'~\cite{matsumoto} that for any closed spin manifold $b_2 \geq \frac{11}{8} |\sigma|$.}. Formula~\eqref{eq:simplyconnectedspin} then follows from multiplicativity of $Z$ and invertibility of $Z(S^2\times S^2)$.
If $\sigma(M) = 16 s \geq 0$, we may replace $K3$ by $\overline{K3}$ in~\eqref{eq:spincaseproof} to get an analogous stable diffeomorphism. It follows from the orientation-preserving diffeomorphism $K3\# \overline{K3} \iso \#^{22} (S^2\times S^2)$ that $Z(K3)$ and $Z(\overline{K3})$ are invertible, and that we may replace $\widetilde{Z}(\overline{K3})$ by $\widetilde{Z}(K3)^{-1} \widetilde{Z}(S^2\times S^2)^{22}$ in the resulting formula. This again results in formula~\eqref{eq:spincaseproof}.
\end{proof}
\begin{remark}
Using Kreck's classification~\cite{Kreck}, analogous formulas can be derived in the non-simply connected case. See Remark~\ref{rem:simplyconnected} for a discussion of this in the simpler case of $\CPt$-stability.
\end{remark}

\section{$\CPt$-stability, the Gluck twist and emergent fermions}\label{sec:fermion}

To lift the spinnability assumption of Corollary~\ref{cor:main} 3. and to better understand the behaviour of field theories on manifolds which admit a spin structure, we turn to the question of $\CPt$-stability of semisimple oriented $4$-dimensional field theories. Since there is an orientation-preserving diffeomorphism $(S^2\times S^2) \#\CPt \iso \CPt \#\CPt\# \overline{\CP}^2$ (see for example~\cite[Exm 5.2.5]{GompfStip}), $\CPt$-stability is a stronger condition than $S^2\times S^2$-stability. And indeed, there are semisimple oriented topological field theories which are not $\CPt$-stable. 

\begin{example} \label{exm:CYCPT}
Recall that the handle diagram for $\CPt$ consists of a single $2$-handle attached along a $1$-framed unknot. In particular, using expression~\eqref{eq:CraneYetterformula}, the Crane-Yetter-Kauffman theory for a ribbon fusion category $\cC$ evaluates $\CPt$ to the \emph{Gauss sum}\[\CY_{\cC}(\CPt) =\cD_{\cC}~ \sum_{X_i \in \mathrm{Irr}(\cC)} \theta_i \dim(X_i)^2\]
where $\theta_i\in k$ denotes the \emph{ribbon twist} of the simple object $X_i$. In particular,  for $\cC$ the symmetric monoidal category $\sVect_k$ of \emph{super vector spaces}\footnote{As a monoidal category, $\sVect_k$ is the category of $\mathbb{Z}_2$-graded vector spaces and grading preserving linear maps. The symmetry isomorphism $\sigma_{V,W}: V\otimes W\to W\otimes V$ is defined in terms of the usual sign rule mapping homogenous vectors $v,w$ with grading $|v|, |w| \in \mathbb{Z}_2$ to $\sigma(v\otimes w) = (-1)^{|v||w|} w\otimes v$. Moreover, we always take $\tVect_k$ to be equipped with the unique ribbon structure for which all dimensions are positive.} with two simple objects $k_+$ and $k_-$ which both have the same dimension $\dim_{\pm} = 1$ but different twists $\theta_{\pm} = \pm 1$, we obtain $\CY_{\sVect_k}(\CPt) =0 $.
Therefore, since $S^2\times S^2$ is $\CPt$-stably diffeomorphic to $\CPt \# \overline{\CP}^2$ and since the indecomposable semisimple oriented topological field theory $\CY_{\sVect}$ vanishes on the latter but is invertible on the former, it follows that $\CY_{\sVect}$ is an example of a field theory which is $S^2\times S^2$-stable but not $\CPt$-stable.
\end{example}

\subsection{The Gluck twist and emergent fermions}
In condensed matter physics, the vanishing of the Gauss sum as in Example~\ref{exm:CYCPT} is more generally anticipated for theories with emergent fermions~\cite{condensedoverflow}\cite[Cor 3.6]{foldway}. In this section, we show that the vanishing of $Z(\CPt)$ is indeed equivalent to the presence of fermions amongst the `point particles' of the topological field theory. Mathematically, this `emergence of fermions' can be characterized as follows.

The \emph{Gluck twist} $\phi_{GT}$ is the diffeomorphism $S^2\times S^1 \to S^2 \times S^1$ defined as $ (x,y) \mapsto (\alpha(y)x,y)$, where $\alpha:S^1\to \mathrm{SO}(3)$ is a representative of the generator of $\pi_1(\mathrm{SO}(3)) =\mathbb{Z}/2\mathbb{Z}$.

\begin{definition} An oriented $4$-dimensional topological field theory $Z: \Bord_{4} \to \Vect_k$ \emph{has emergent fermions} if the Gluck twist $Z(\phi_{GT})$ acts non-trivially.
\end{definition}

\begin{remark}\label{rem:fermion}
Recall from Remark~\ref{rem:physics} that from the perspective of physics, the `fusion algebra' $Z(S^2\times S^1)$ may be understood as a decategorification, or trace, of the category (or more generally, object of some $2$-category) of `point particles' $Z(S^2)$ in an extended $4$-dimensional field theory. From this perspective, the Gluck twist, seen as an invertible $4$-dimensional bordism with corners $S^2\times [0,1] \To S^2\times [0,1]$ may be understood as encoding the operation of rotating a point particle by 360 degrees in $3$-space. Indeed, we point out that even though our field theory is `bosonic' in that it is a functor from an oriented (and not, say, spin) bordism category to the category of ordinary (and not, say, super) vector spaces, it may nevertheless have fermionic point particles which do behave non-trivially under this action of $\pi_1(\mathrm{SO}(3)) = \mathbb{Z}/2\mathbb{Z}$. In the condensed-matter physics community this phenomena is known as the `emergence of fermions' in a bosonic topological order~\cite{condensed}. In particular, we may decompose the vector space of `point particles' $Z(S^2\times S^1)$ into a vector space of bosons $Z(S^2\times S^1)_+$ on which the Gluck twist acts trivially and a vector space of fermions $Z(S^2\times S^1)_-$ on which the Gluck twist acts as minus the identity. 
\end{remark}

\begin{example}\label{exm:CYfermion} For the Crane-Yetter-Kauffman theory of Examples~\ref{exm:CYone} and~\ref{exm:CYCPT}, the Gluck twist $\CY_{\cC}(\phi_{GT}): \CY_{\cC}(S^2\times S^1) \to \CY_{\cC}(S^2\times S^1)$ is the endomorphism of $K_0(\Zsym(\cC))\otimes_{\mathbb{Z}}k$ mapping (the isomorphism class of) a transparent simple object $X_i$ to $\theta_i X_i$, where $\theta_i \in k^\times$ denotes the ribbon twist of $X_i$. In particular, in agreement with Remark~\ref{rem:fermion}, the Gluck twist acts non-trivially if and only if the ribbon category $\cC$ has \emph{fermions}\footnote{In our terminology, a fermion in a ribbon fusion category is a transparent simple object $f$ with non-trivial twist $\theta_f=-1$. This differs somewhat from the terminology in parts of the literature~\cite{foldway} where fermions are often also required to fulfill $f\otimes f \iso \mathrm{I}$. }  --- transparent simple objects with non-trivial twist.
\end{example}

\subsection{The Gluck twist is $\CPt$-trivial}
In the following, we show that the Gluck twist, seen as an invertible $4$-dimensional bordism $S^2\times S^1 \to S^2\times S^1$, is $\CPt$-stably diffeomorphic to the cylinder $S^2\times S^1\times[0,1]$. Hence, any $\CPt$-stable oriented topological field theory has trivial Gluck twist, and can therefore not have emergent fermions. 

For an orientation-preserving diffeomorphism $\phi: M\to N$ of closed oriented $3$-manifolds, we write $\mathrm{Cyl}(\phi):M\to N$ for the mapping cylinder, defined as the compact oriented $4$-dimensional bordism $M\xhookrightarrow{\phi \times \{0\}} N\times [0,1] \xhookleftarrow{\id_N \times \{1\}} N$. Given a topological field theory, we abuse notation and write $Z(\phi):Z(M)\to Z(N)$ for the linear map $Z(\mathrm{Cyl}(\phi))$.

All results of this section are direct consequences of the following diffeomorphism.

\begin{prop} \label{prop:diffeoCP1}There is an orientation-preserving diffeomorphism of the following oriented bordisms $S^2\times S^1 \to S^2\times S^1$:
\begin{equation}\label{eq:CPtdiffeo} \CPt\#\mathrm{Cyl}(\phi_{GT})  \iso \CPt\#(S^2\times S^1\times [0,1])
\end{equation}
\end{prop}
\begin{proof} We construct a diffeomorphism $\psi$ of $\CPt \# \left(S^2\times S^1\times[0,1]\right)$ such that $\psi|_{S^2\times S^1\times \{0\}}$ is the identity and $\psi|_{S^2\times S^1\times \{1\}}$ is the Gluck twist $\phi_{GT}$. Recall that a \emph{relative Kirby diagram}~\cite[Sec 5.5]{GompfStip} for a connected compact oriented $4$-bordism $W:\partial_- W\to \partial_+ W$ with connected and non-empty $\partial_- W$ comprises a surgery diagram for the $3$-manifold $\partial_- W$, together with a Kirby diagram for the additional handles of $W$ superimposed on that surgery diagram. Following the convention of~\cite[Sec 5.5]{GompfStip}, we put brackets around the framing coefficients of the link components belonging to the surgery diagram of $\partial_-W$. Therefore, given a relative Kirby diagram, the $3$-manifold $\partial_-W$ can be obtained by doing surgery on the sublink with bracketed framing coefficients, and $\partial_+W$ can be obtained by doing surgery on the entire link diagram.

A relative Kirby diagram for $W=\CPt\#(S^2\times S^1 \times[0,1])$ consists of the disjoint union of a $\langle 0 \rangle $-framed unknot (a surgery diagram for $S^2\times S^1$) and a $1$-framed unknot (a $2$-handle attachment giving rise to the bordism $\CPt \#(S^2\times S^1\times[0,1]$). 
We construct $\psi$ as a sequence of relative Kirby moves~\cite[Thm 5.5.3]{GompfStip} such that, when restricted to bracketed components, $\psi$ is the identity and, when considered as a sequence of moves of surgery diagrams for $\partial_+W = S^2\times S^1$, $\psi$ corresponds to the Gluck twist. 
Explicitly, we define $\psi$ to be the handle slide of the $1$-framed unknot over the $\langle 0 \rangle$-framed unknot.  When restricted to $\partial_-W = S^2\times S^1$, $\psi$ is clearly the identity. When restricted to $\partial_+W$, $\psi$ defines an orientation-preserving diffeomorphism $\psi_+$ of $S^2\times S^1$. On the obvious embedding $D^2 \times S^1 \hookrightarrow S^2 \times S^1$, $\psi_+$ restricts to a diffeomorphism of $D^2 \times S^1$ which changes framing by one. Hence, it follows that $\psi_+$ acts trivially on $\pi_1(S^2 \times S^1)$. However, following the prescription of~\cite[Exm 5.5.8]{GompfStip} for gluing $S^2\times D^2$ along $\psi_+$ to $S^2\times D^2$ results in the non-trivial $S^2$-bundle $S^2\widetilde{\times} S^2$. Hence, $\psi_+$ is not  isotopic to the identity. But it follows from~\cite{Gluck} that up to isotopy, the Gluck twist $\phi_{GT}$ is the unique non-trivial orientation-preserving diffeomorphism of $S^2 \times S^1$ which acts trivially on $\pi_1(S^2 \times S^1)$. Therefore, $\psi_+$ is isotopic to the Gluck twist $\phi_{GT}$.
\end{proof}

\begin{corollary} \label{cor:glucktrivial}Let $Z:\Bord_4\to \Vect_k$ be a (not necessarily semisimple) $\CPt$-stable oriented $4$-dimensional topological field theory. Then, the Gluck twist $Z(\phi)$ acts trivially, that is there are no emergent fermions.
\end{corollary}
\begin{proof}By Proposition~\ref{prop:diffeoCP1}, the connected bordisms $\mathrm{Cyl}(\phi)$ and $S^2\times S^1 \times [0,1]$ are $\CPt$-stably diffeomorphic. Hence, $Z(\phi) = \id_{Z(S^2\times S^1)}$.
\end{proof}

\begin{example} Since any invertible topological field theory $Z$ is multiplicative on connected sums (up to $Z(S^4)$-factors) and invertible on $\CPt$, it is automatically $\CPt$-stable and can therefore not have emergent fermions.
\end{example}
\subsection{Semisimple field theories are $\CPt$-stable iff they have trivial Gluck twist}\label{sec:mainfermion}
In this section, we show that for semisimple field theories, the converse of Corollary~\ref{cor:glucktrivial} is true.\looseness=-2

\begin{theorem}\label{thm:emergentfermions} Let $Z$ be an indecomposable semisimple oriented $4$-dimensional topological field theory. Then, the following are equivalent:
\begin{enumerate}
\item $Z$ is $\CPt$-stable.
\item$Z$ has no emergent fermions, that is, the Gluck twist acts trivially;
\item One, or equivalently both of $Z(\CPt)$ and $Z(\overline{\CP}^2)$ are invertible;
\end{enumerate}
\end{theorem}
\begin{proof} 
Corollary~\ref{cor:glucktrivial} shows $1. \Rightarrow 2.$ for all oriented $4$-dimensional topological field theories.

To prove $2.\Rightarrow 3.$, we note that $\CPt\# \overline{\CP}^2$ can be obtained from $S^2\times S^2$ by Gluck surgery (that is by removing an embedded $S^2\times D^2$ and gluing it back in via the Gluck twist). Hence, triviality of the Gluck twist implies that $Z(S^2\times S^2)= Z(\CPt \#\overline{\CP}^2)$. Since $Z(S^2\times S^2)$ is invertible (Theorem~\ref{thm:S2invertible}) and $Z$ is indecomposable and hence multiplicative under connected sums (Proposition~\ref{prop:multiplicativityconnectsum}), it follows that $Z(\CPt)$ and $Z(\overline{\CP}^2)$ are invertible.

If one of $Z(\CPt)$ and $Z(\overline{\CP}^2)$ is invertible, multiplicativity, invertibility of $Z(S^2\times S^2)$ and the orientation-preserving diffeomorphism $\overline{\CP}^2\#2\CPt \iso (S^2\times S^2) \#\CPt$ (or its orientation-reversal) imply that the other is also invertible. Hence, $3. \Rightarrow 1.$ follows immediately from multiplicativity of $Z$.
\end{proof}

\begin{remark} For indecomposable semisimple $4$-dimensional field theories $Z$, Theorem~\ref{thm:emergentfermions} implies that instead of requiring it for all connected oriented bordisms, $\CPt$-stability is equivalent to the --- a priori weaker --- condition that $Z(M) = Z(N)$ for all $\CPt$-stably diffeomorphic closed oriented $4$-manifolds $M$ and $N$. Indeed, since $S^2\times S^2$ and $\CPt \# \overline{\CP}^2$ are $\CPt$-stably diffeomorphic, $\CPt$-stability on closed manifolds implies that $Z(S^2\times S^2) = Z(\CPt \# \overline{\CP}^2)$ and hence that both $Z(\CPt)$ and $Z(\overline{\CP}^2)$ are invertible. 
\end{remark}
The proof of Theorem~\ref{thm:mainfermion} follows immediately from Theorem~\ref{thm:emergentfermions} by decomposing a semisimple field theory into its indecomposable components.
\begin{proof}[Proof of Theorem~\ref{thm:mainfermion}] By Corollary~\ref{cor:glucktrivial}, the Gluck twist acts trivially in any $\CPt$-stable oriented $4$-dimensional topological field theory. Conversely, suppose that $Z$ is a semisimple oriented topological field theory with trivial Gluck twist. By Proposition~\ref{prop:decomposess}, $Z \iso \bigoplus_i Z_i$ can be decomposed into a finite direct sum of indecomposable semisimple field theories $Z_i$ on which the Gluck twist still acts trivially.  Therefore, it follows from Theorem~\ref{thm:emergentfermions} that every component $Z_i$ is $\CPt$-stable, and hence so is $Z$.
\end{proof}

\subsection{Consequences of Theorem~\ref{thm:mainfermion}}\label{sec:fermionapp}
The classification of $\CPt$-stable diffeomorphism types is considerably simpler than in the $S^2\times S^2$-case. Indeed, it follows from Kreck's work that two closed oriented $4$-manifolds are $\CPt$-stably diffeomorphic if and only if they have the same Euler characteristic, signature, isomorphic fundamental groups and if the images $c_*[M]\in \mathrm{H}_4(\pi_1, \mathbb{Z})$ of their fundamental class under a classifying map $c:M\to K(\pi_1, \mathbb{Z})$ of their universal cover agree (see e.g.~\cite{teichner}).

\begin{proof}[Proof of Corollary~\ref{cor:mainfermion}] The data described in the last paragraph only depends on the oriented homotopy type of $M$.
\end{proof}

Similar to the $S^2\times S^2$-case, we may evaluate $Z$ on some $4$-manifold $M$ by evaluating it on a simpler `reference manifold' in the same $\CPt$-stable diffeomorphism class. In particular, we immediately obtain Corollary~\ref{cor:explicitfermion} as a refinement of Corollary~\ref{cor:explicitformula}.

\begin{proof}[Proof of Corollary~\ref{cor:explicitfermion}] We first consider the case that $Z$ does not have emergent fermions, that is that the Gluck twist acts trivially. In this case, it follows from Theorem~\ref{thm:mainfermion} that $Z$ is $\CPt$-stable.  By the above discussion, two connected, simply connected closed oriented $4$-manifolds are $\CPt$-stably diffeomorphic if and only if they have the same Euler characteristic and signature. In particular, any such $M$ is $\CPt$-stably diffeomorphic to $\#^{b_2^+(M)}\CPt \#^{b_2^-(M)}\overline{\CP}^2$. The `no fermions' entry of the table follows from multiplicativity of the indecomposable semisimple field theory $Z$. Invertibility of this entry follows from Theorem~\ref{thm:emergentfermions} 3.

If $Z$ has emergent fermions, it follows from Theorem~\ref{thm:emergentfermions} 3. that $\widetilde{Z}(\CPt) = \widetilde{Z}(\overline{\CP}^2) = 0$ and hence, by Corollary~\ref{cor:explicitformula}, that $M$ vanishes on simply connected manifolds which do not admit a spin structure. The value of $Z$ on simply connected spin manifold follows from Corollary~\ref{cor:explicitformula}. It is invertible, since $Z(S^2\times S^2)$ is invertible and since there is an orientation-preserving diffeomorphism $K3\#\overline{K3}\iso \#^{22} (S^2\times S^2)$.
\end{proof}

\begin{remark}\label{rem:simplyconnected} More generally, as explained in~\cite[Sec 1.4]{teichner}, if $M$ and $N$ are closed oriented $4$-manifolds with the same fundamental group $\pi$ and class $c_*[M] = c_*[N] \in H_4(\pi_1)$ and $Z$ is an indecomposable semisimple field theory without emergent fermions, it follows that\looseness=-2
\[\widetilde{Z}(M)= \widetilde{Z}(N)~\widetilde{Z}(\CPt)^{\frac{1}{2}\left( \Delta \chi + \Delta \sigma\right)}~\widetilde{Z}(\overline{\CP}^2)^{\frac{1}{2}\left(\Delta \chi - \Delta \sigma\right)}
\]where $\Delta\chi = \chi(M) - \chi(N)$ and $\Delta \sigma = \sigma(M)- \sigma(N)$. For example, taking $N=S^1\times S^3$ and noting that $Z(S^1\times S^3) = \dim Z(S^3) = 1$, this allows to evaluate the invariant on manifolds with fundamental group $\ZZ$ as follows:
\[\widetilde{Z}(M) = Z(S^4)^{-1}~\widetilde{Z}(\CPt)^{\frac{1}{2}\left( \chi(M)+ \sigma(M)\right)} \widetilde{Z}(\overline{\CP}^2)^{\frac{1}{2} \left( \chi(M)-\sigma(M)\right)}
\]
\end{remark}

\begin{example} As discussed in Example~\ref{exm:CYfermion}, the Crane-Yetter-Kauffman field theory $\CY_{\cC}$ has fermions if and only if $\cC$ has a transparent simple object with non-trivial twist. In particular, if $\cC$ is a modular tensor category, $\CY_{\cC}$ does not have fermions since the only transparent simple object of $\cC$ is the tensor unit $I$. Therefore, the non-fermion case of Corollary~\ref{cor:explicitfermion} may be seen as a generalization of the explicit computation of $\CY_{\cC}$ in~\cite{CYKCompute}. Of course, in this particular case, this explicit expression follows more directly from the fact that modularity of $\cC$ is equivalent to invertibility of the field theory $\CY_{\cC}$, and the classification of $4$-dimensional invertible oriented field theories in terms of Euler characteristic and signature~\cite[Sec 1.3]{SchommerPriesTori}.
\end{example}

Applying Theorem~\ref{thm:emergentfermions} to the Crane-Yetter-Kauffman theory results in a topological proof of the following algebraic characterization of ribbon fusion categories with fermions.
\begin{cor}\label{cor:Gausssum} Let $\cC$ be a ribbon fusion category over an algebraically closed field of characteristic zero. Then, $\cC$ contains a simple transparent object with non-trivial twist if and only if one, or equivalently both, of the Gauss sums $\tau_{\cC}^\pm:=\sum_{X_i \in \mathrm{Irr}(\cC)} \theta^{\pm 1}_i \dim(X_i)^2$ are zero.\looseness=-2
\end{cor}
\begin{proof}[Proof of Corollary~\ref{cor:Gausssum}]Since $\CY_{\cC}$ is an indecomposable semisimple field theory which has emergent fermions if and only if $\cC$ has a simple transparent object with non-trivial twist (Example~\ref{exm:CYfermion}), and since $\CY_{\cC}(\CPt)$ and $\CY_{\cC}(\overline{\CP}^2)$ are proportional to the Gauss sums (Example~\ref{exm:CYCPT}), Corollary~\ref{cor:Gausssum} follows from the equivalence $2\Leftrightarrow 3$ in Theorem~\ref{thm:emergentfermions}.
\end{proof}

\begin{remark} An equivalent, purely algebraic proof of Corollary~\ref{cor:Gausssum} can be obtained starting from the following formula~\cite[Lem 6.10]{EGNO} which holds for any simple object $Y$ in a ribbon fusion category $\cC$:
\begin{equation}\label{eq:EGNO}\theta_Y \sum_{X \in \mathrm{Irr}(\cC)} \theta_{X}\dim(X) s_{X,Y} = \dim(Y)\tau_{\cC}^+
\end{equation}
Here, $\theta_X$ and $\dim(X)$ denote the twist and dimension of a simple object $X$, respectively, and $s_{X,Y}$ denotes the value of the `$S$-matrix' of $\cC$; the evaluation of a ($0$-framed) Hopf link whose components are labelled by the simple objects $X$ and $Y$. As a direct consequence of~\eqref{eq:EGNO} (and as a special case of~\cite[Prop 6.11]{EGNO}) it follows that
\begin{equation}\label{eq:algS2}\tau_{\cC}^+ \tau_{\cC}^- = \dim(\cC) \tau_{\Zsym(\cC)}^+
.\end{equation} In particular, if every transparent object $Y$ of $\cC$ has trivial twist $\theta_Y = 1$, then $\tau_{\Zsym(\cC)}^+ = \dim(\Zsym(\cC))$ and hence $\tau_{\cC}^+\tau_{\cC}^- = \dim(\cC) \dim(\Zsym(\cC))$ is non-zero. Conversely, if $Y$ is a simple transparent object, it follows from~\eqref{eq:EGNO} and invertibility of $\dim(Y)$ that \begin{equation}\label{eq:alg2}\theta_Y \tau_{\cC}^+ = \tau_{\cC}^+.\end{equation} Therefore, the existence of a simple transparent object $Y$ with twist $\theta_Y =-1$ implies that $\tau_{\cC}^+=0$, and similarly that $\tau_{\cC}^-=0$.
Note that equations~\eqref{eq:algS2} and~\eqref{eq:alg2} may be understood as algebraic analogues of the diffeomorphism between $\CPt\# \overline{\CP}^2$ and the non-trivial $S^2$-bundle $S^2\widetilde{\times} S^2$, and of Proposition~\ref{prop:diffeoCP1}, respectively.
\end{remark}

\bibliographystyle{initalpha}
\bibliography{SS4D}

\end{document}

%% file: SS4D-preamble.tex
\usepackage{amsmath}       
\usepackage{amsthm}        
\usepackage{amssymb}       
\usepackage{amsfonts}
\usepackage[all]{xy}
\usepackage{xspace}
\usepackage{calc}
\usepackage{pgfplots}
\usepgflibrary{fpu}
\usepackage{mathtools}
\usepackage{tabularx}
\usepackage{ifthen}		
\usepackage{graphicx}
\usepackage{docmute}
\usepackage{array}
\usepackage{subfloat} 
\usepackage{bbm}            
\usepackage{etoolbox}  
\usepackage{verbatim}
\usepackage{stmaryrd}
\usepackage{thm-restate}
\usepackage{multicol}
\usepackage{nccmath}

\usepackage[margin=1.5in]{geometry}
\usepackage{tikz}
\usetikzlibrary{matrix}
\usetikzlibrary{decorations.pathreplacing}
\usetikzlibrary{decorations.markings}
\usetikzlibrary{arrows}
\usetikzlibrary{calc}
\usetikzlibrary{cd}
\usetikzlibrary{knots}
\usetikzlibrary{shapes.misc}
\usetikzlibrary{fit}
\usepgflibrary{decorations.pathmorphing}
\usepgflibrary{shapes.geometric}
\newenvironment{tz}[1][]{%
			\begin{tikzpicture}[baseline={([yshift=-.8ex]current bounding 					box.center)},#1] %
				}{%
			\end{tikzpicture} %
			}
			
\tikzset{dot/.style={circle, scale=0.3, fill=black, thick, draw}}
\tikzset{knot/.style={draw=white, double distance=1, line width=0.5, double=black}}

\renewcommand{\to}[1][]{\ensuremath{\xrightarrow{#1}}}
\newcommand{\To}[1][]{\ensuremath{\xRightarrow{#1}}}

\allowdisplaybreaks[1]

\theoremstyle{plain} 
\newtheorem{maintheorem}{Theorem}

\newtheorem{maincor}[maintheorem]{Corollary}

\newtheorem{theorem}{Theorem}[section]

\newtheorem{corollary}[theorem]{Corollary}          
              
\newtheorem{cor}[theorem]{Corollary}          
\newtheorem{prop}[theorem]{Proposition}

\theoremstyle{definition} 
\newtheorem{definition}[theorem]{Definition}

\theoremstyle{remark}  

\newtheorem{remark}[theorem]{Remark}
\newtheorem{example}[theorem]{Example}

\newtheorem{question}[theorem]{{Question}}
\newtheorem*{guide*}{Guide}
\newtheorem*{outline*}{Outline}
\newtheorem*{remarkohc*}{Remark on higher categories and the cobordism hypothesis}
\newtheorem*{assumpfield*}{Assumptions on the base field}

\DeclareMathOperator{\Hom}{Hom}
\DeclareMathOperator{\End}{End}

\newcommand{\id}{\mathrm{id}}
\newcommand{\iso}{\cong}

\newcommand{\Bord}{\mathrm{Bord}}

\newcommand{\Hilb}{\mathrm{Hilb}}

\newcommand{\Vect}{\mathrm{Vect}}
\newcommand{\tVect}{\mathrm{2Vect}}

\def\cB{\mathcal B}\def\cC{\mathcal C}\def\cD{\mathcal D}

\def\ZZ{\mathbb Z}

\usepackage[pdftex, colorlinks=true, linkcolor=black, citecolor=black, urlcolor=black, hypertexnames=false,
pdfauthor={David J. Reutter},
pdftitle={\titlepaper}
	]{hyperref}

\newcommand{\CPt}{\mathbb{C}P^2}
\newcommand{\CP}{\mathbb{C}P}

\newcommand{\sVect}{\mathrm{sVect}}
\newcommand{\CY}{\mathrm{CYK}}
\newcommand{\Zsym}{\mathrm{Z}_{\mathrm{sym}}}